\documentclass[11pt]{extarticle}

\usepackage[left=1in,right=1in,top=1in,bottom=1.3in]{geometry}
\usepackage{amsmath,amssymb,amsthm,latexsym,amsfonts,indentfirst,xcolor,mathtools}
\usepackage[utf8]{inputenc} 
\usepackage[english]{babel}
\usepackage[breaklinks]{hyperref}
\hypersetup{
	colorlinks = true, 
	urlcolor = cyan, 
	linkcolor = teal, 
	citecolor = cyan 
}

\usepackage{float}

\let\OLDthebibliography\thebibliography
\renewcommand\thebibliography[1]{
	\OLDthebibliography{#1}
	\setlength{\parskip}{0pt}
	\setlength{\itemsep}{0pt plus 0.3ex}
}

\newtheorem{Corollary}{Corollary}

\newtheorem{Lemma}{Lemma}
\newtheorem{Proposition}{Proposition}
\newtheorem{Theorem}{Theorem}

\newcommand{\R}{{\mathbb R}}
\newcommand{\Z}{{\mathbb Z}}
\newcommand{\N}{{\mathbb N}}

\makeatletter

\def\env@sqcases{%
	\let\@ifnextchar\new@ifnextchar
	\left\lbrack
	\def\arraystretch{1.2}%
	\array{@{}l@{\quad}l@{}}%
}
\makeatother

\begin{document}

\date{}
\title{Solution to a conjecture of Schmidt and Tuller \\ on one-dimensional packings and coverings}
\author{
	N\'ora Frankl\thanks{Alfréd Rényi Institute of Mathematics, Budapest, Hungary. Supported by ERC Advanced Grant `GeoScape'. Email:~\href{mailto:nfrankl@renyi.hu}{\tt nfrankl@renyi.hu}. },
	Andrey Kupavskii\thanks{MIPT, Moscow, Russia and G-SCOP, Universit\'e Grenoble-Alpes, CNRS, France. Research supported by the \href{https://rscf.ru/project/21-71-10092/}{RSF grant N 21-71-10092}. Email:~\href{mailto:kupavskii@ya.ru}{\tt kupavskii@ya.ru}.},
	Arsenii Sagdeev\thanks{MIPT, Moscow, Russia and Alfréd Rényi Institute of Mathematics, Budapest, Hungary. Supported in part by ERC Advanced Grant `GeoScape'. The author is also a winner of Young Russian Mathematics Contest and would like to thank its sponsors and jury. Email:~\href{mailto:sagdeevarsenii@gmail.com}{\tt sagdeevarsenii@gmail.com}.}
}

\maketitle

\begin{abstract}
	In 2008, Schmidt and Tuller stated a conjecture concerning optimal packing and covering of integers by translates of a given three-point set.  In this note, we confirm their conjecture and relate it to several other problems in combinatorics.
\end{abstract}

\section{Introduction}

Let $S$ be a subset of $\mathbb Z$. We call $S' \subset \Z$ a {\it translated copy} (or a {\it translate}) {\it of $S$} if for some $t \in \Z$ we have $S' = S+t$. In this note, we study packings in and coverings of $\mathbb Z$ by translates of a finite set $S$. 
We also work with two related notions: $S$-free and $S$-blocking sets. We define all these four notions below. Given $A\subset \Z$, recall that the \emph{upper density $\overline{d}(A)$}  is defined as $\overline{d}(A)=\limsup_{n\to \infty}\frac{|A\cap[-n,n]|}{2n+1}$. Similarly, the \emph{lower density $\underline{d}(A)$} is defined as $\underline{d}(A)=\liminf_{n\to \infty}\frac{|A\cap[-n,n]|}{2n+1}$. Further, if $\overline{d}(A)=\underline{d}(A)$, then we define the \emph{density $d(A)$} of $A$ as $d(A) \coloneqq \overline{d}(A)=\underline{d}(A)$.

\begin{itemize}

\item A subset $A \subset \Z$ is called {\it $S$-packing} if translates $S+a_1$ and $S+a_2$ are disjoint for all distinct $a_1,a_2 \in A$. Define a \emph{maximum packing density $d_p(S)$ of $S$} as the supremum of the upper densities $\overline{d}(A)$ of $S$-packing subsets $A\subset \Z$.

\item A subset $A \subset \Z$ is called {\it $S$-covering} if each $t \in \Z$ belongs to some translate $S+a$, $a \in A$. Define a \emph{minimum covering density $d_c(S)$ of $S$} as the infimum of the lower densities $\underline{d}(A)$ of $S$-covering subsets $A\subset \Z$.

\item A subset $A \subset \Z$ is called {\it $S$-free} if $A$ does not contain any translate of $S$. We denote by $d_f(S)$ the supremum of the upper densities $\overline{d}(A)$ of $S$-free subsets $A\subset \Z$.

\item A subset $A\subset \Z$ is called \emph{$S$-blocking} if $A$ intersects every translate of $S$. We denote by $d_b(S)$ the infimum of the lower densities $\underline{d}(A)$ of $S$-blocking subsets $A\subset \Z$.

\end{itemize}

It is not hard to check that for all $A,S \subset \Z$ the following three conditions are equivalent: the {\it reflection}  $-A = \{-a: a \in A\}$ is $S$-covering, $A$ is $S$-blocking, and the compliment $\Z\!\setminus\!A$ is $S$-free, see~\cite{Axe} for details. Hence, for any finite $S \subset \Z$, we have
\begin{equation} \label{eq_cbf}
	d_c(S) = d_b(S) = 1-d_f(S).
\end{equation}

Coverings of the integers by translates of infinite sets were considered by Erd\H os~\cite{Erd54} and Lorenz~\cite{Lor}, but probably the first to study minimum covering densities of finite subsets of $\Z$ was Newman~\cite{New}. It is clear that if $|S|=2$, then $d_p(S)=d_c(S) = \frac{1}{2}$. However, even for  $3$-point sets, the situation becomes much more complex. In particular, Newman~\cite{New} showed that $d_c(S)\le \frac 25$ for any $3$-element subset $S \subset \Z$ and this is tight if $S=\{0,1,3\}$. 

For general sets, Newman \cite{New} showed that $d_c(S)\le (1+o(1))\frac{\log k}{k}$ for any $k$-element set $S,$ where $k\to \infty$. Moreover, this bound is asymptotically tight for certain $k$-element sets. A sharper asymptotic result in a more general setting was obtained by Bollob\'as, Janson and Riordan~\cite[Theorem 4.5]{BJR}. Bounds on packings are not so sharp: given $k \in \N$, it is known that $d_p(S)\ge \frac{2}{k^2}$ for all $k$-element sets $S$, as well as that there is a set $S$ with $d_p(S)\le \frac{2.646}{k^2},$ see~\cite{Wein, Gol}. 

Our main result confirms the conjecture of Schmidt and Tuller~\cite{Schtu} that suggested certain explicit expressions for packing and covering densities for $3$-point sets. Some particular cases were resolved by Schmidt and Tuller themselves, as well as in the paper~\cite{Ste} by Stein from 1986, twenty years before the conjecture was formulated. At the heart of our proof is an intricate double-counting argument that may be of independent interest.

\begin{Theorem}[Schmidt--Tuller conjecture for packings and coverings\footnote{While the first version of this paper was under review, the authors found out that the first (`packing') half of the Schmidt--Tuller conjecture is equivalent to a graph-theoretical problem considered by Rabinowitz and Proulx~\cite[Conjecture 5.5]{RabPr} in 1985. Liu and Zhu~\cite{LZ04} confirmed this conjecture in 2004, four years before Schmidt and Tuller stated it independently in terms of packing densities. Our proof is different and treats packings and coverings in a unified way. 
}]\label{thm:3baton}
	Let $\lambda_1,\lambda_2 \in \N$ be two coprime integers. Then for $S=\{0,\lambda_1,\lambda_1+\lambda_2\}$, we have
	\begin{equation*}
		d_p(S) = \max \left (\frac{\lfloor \frac{1}{3}( \lambda_1+2\lambda_2)\rfloor}{\lambda_1+2\lambda_2},\frac{\lfloor \frac{1}{3}( 2\lambda_1+\lambda_2)\rfloor}{2\lambda_1+\lambda_2}\right),
	\end{equation*}
	and
	\begin{equation*}
		d_c(S) = \min \left (\frac{\lceil \frac{1}{3}( \lambda_1+2\lambda_2)\rceil}{\lambda_1+2\lambda_2},\frac{\lceil \frac{1}{3}( 2\lambda_1+\lambda_2)\rceil}{2\lambda_1+\lambda_2}\right).
	\end{equation*}
\end{Theorem}

It is not hard to check that one can rewrite the statement of Theorem~\ref{thm:3baton} in the following way:
\begin{equation*} \small{
	d_p(S) =
	\begin{cases}
		1/3 & \mbox{if } \lambda_1 \equiv \lambda_2 \hspace{17pt} \pmod 3 ;\\[3pt]
		\frac{2\lambda_1+ \lambda_2 - 1}{3(2\lambda_1+\lambda_2)} & \mbox{if } \lambda_1 \equiv \lambda_2-1 \pmod 3 ;\\[3pt]
		\frac{\lambda_1+ 2\lambda_2 -1}{3(\lambda_1+2\lambda_2)} & \mbox{if } \lambda_1 \equiv \lambda_2+1 \pmod 3, \hspace{20pt}
	\end{cases}
	d_c(S) =
	\begin{cases}
		1/3 & \mbox{if } \lambda_1 \equiv \lambda_2 \hspace{17pt} \pmod 3 ;\\[3pt]
		\frac{2\lambda_1+ \lambda_2 +1}{3(2\lambda_1+\lambda_2)} & \mbox{if } \lambda_1 \equiv \lambda_2+1 \pmod 3 ;\\[3pt]
		\frac{\lambda_1+ 2\lambda_2 +1}{3(\lambda_1+2\lambda_2)} & \mbox{if } \lambda_1 \equiv \lambda_2-1 \pmod 3.
	\end{cases} }
\end{equation*}
While the proofs for coverings and packings in Theorem~\ref{thm:3baton} are similar, we did not manage to find a clean unified argument, so we spell them out separately in Sections~\ref{sec:cov} and~\ref{sec:pack} respectively.

Note that no generality is lost because of the assumption  that $\lambda_1,\lambda_2$ are relatively prime. Indeed, if $S$ is as in the Theorem \ref{thm:3baton} and $S' = \{0, m\lambda_1, m(\lambda_1+\lambda_2)\}$ for some integer $m$, then $d_c(S) = d_c(S').$ To see this, note that each translate of $S'$ entirely lies in $m\Z+i$ for some $i\in [m]$. Thus, when covering by $S'$, we have to cover each $m\Z+i,$ $i\in [m],$ individually. For each of these subproblems, we recover the original problem of covering $\Z$ with translates of $S$ (after scaling down by a common factor of $m$). The same argument works for packing densities.

In the next subsection, we state a slightly stronger version of the second part of Theorem~\ref{thm:3baton}, as well as some related results.

\subsection{Related results}

Let $d_f(S,-S)$ be the supremum of the upper densities $\overline{d}(A)$ of subsets $A\subset \Z$ that avoid translates of \emph{both} $S$ and $-S$. Taking~\eqref{eq_cbf} into account, the second part of Theorem~\ref{thm:3baton} can be strengthened as follows.

\begin{Theorem}[A stronger version of Theorem \ref{thm:3baton} for coverings]\label{thm:-S} Let $S$ be as in Theorem~\ref{thm:3baton}. Then
	\begin{equation*}
		d_f(S,-S)=d_f(S)= \max \left (\frac{\lfloor \frac{2}{3}( \lambda_1+2\lambda_2)\rfloor}{\lambda_1+2\lambda_2},\frac{\lfloor \frac{2}{3}( 2\lambda_1+\lambda_2)\rfloor}{2\lambda_1+\lambda_2}\right).
	\end{equation*}
\end{Theorem}

It should be clear that $d_f(S,-S)\leq d_f(S)$. Thus, in order to prove Theorem~\ref{thm:-S} (and obtain the second part of Theorem~\ref{thm:3baton} as a corollary), it is sufficient to do the following. First, get the lower bound on $d_f(S,-S)$ via explicit construction of a subsets $A\subset \Z$ that avoid translates of both $S$ and $-S$, and then prove the upper bound on $d_f(S)$ that meets the lower one.

The reason we spell out  Theorem~\ref{thm:-S} separately is in its connection to certain Ramsey-theoretic questions. Let $\R^n_{\infty}$ be the $n$-space equipped with the maximum norm. The second and third authors managed to prove~\cite{KS, KS2} that for any finite metric space $\mathcal{M}$, the minimum number of colours needed to colour the points of $\R^n_{\infty}$ such that no isometric copy of $\mathcal{M}$ is monochromatic exponentially tends to infinity with $n$. Besides, they qualitatively  reduced this general question to the question of determining $d_f(S,-S)$ for certain linear sets. In a follow-up paper~\cite{FKS}, the authors of the present note made this connection much more explicit and tight.

Schmidt and Tuller~\cite{Schtu} showed that the optimal packing (and covering) density of $S$ is attained on some periodic set, with a period bounded by $\textrm{diam}(S)(2^{\textrm{diam}(S)}+1)$. Hence, these densities are always rational, and can be computed for any given finite $S$. Interestingly, it is an open question whether an analogue of this is true for packings and coverings of $\Z^n$ for $n\geq 2$, see~\cite{SchTul2010}. One known result in this direction is that if $S$ tiles $\mathbb{Z}^2$, then there is a tiling  by a periodic set~\cite{Bhat2020}. 

Several other problems concerning tilings have been studied. For example, the question of characterising the sets $S\subset \Z$ that tile $\Z$ has received a lot of attention. Using the fact that optimal packings and coverings are attained on periodic sets, this is the same as asking when does $d_c(S) = d_p(S) = |S|^{-1}$ hold. Axenovich et al.~\cite[Corollary 15]{Axe} proved an even stronger statement: a finite subset $S\subset \Z$ tiles $\Z$ if and only if $d_f(S,-S) = 1-|S|^{-1}$ holds. Though the general question of characterisating such sets $S$ remains open already in dimension one, Newman~\cite{New2} managed to get an elegant solution for the case when $|S|$ is a power of a prime number. His result was extended to $|S| = p_1^{\alpha_1} p_2^{\alpha_2}$ by Coven and Meyerowitz~\cite{CM}, although the characterisation is more technical. Both their proofs were algebraic and based on studying the properties of the cyclotomic polynomials. Newman~\cite{New2} asked for a different, `trivial' proof for the case $|S|=3$. We note that our Theorem~\ref{thm:3baton} provides if not `trivial', then at least purely combinatorial solution of this problem.

Concerning other small sets,  Axenovich et al.~\cite{Axe} recently showed via a computer-aided proof that $d_c(S)\le \frac 13$ for any $4$-element set $S\subset \Z$, confirming a conjecture of Newman~\cite{New}. This bound is tight for the set $S=\{0,1,2,4\}$. Some partial results on this conjecture were previously obtained by Weinstein~\cite{Wein} and Bollob\'as, Janson and Riordan~\cite{BJR}. In the latter paper there are also conjectures about the best possible general upper bounds on the covering densities for $5$- and $6$-element sets.

Finally, we mention that $S$-blocking and $S$-free sets are related to polychromatic colourings of the integers. For a given $S,$ a colouring of $\Z$ is {\it polychromatic} if every translate of $S$ contains an element of each colour. The {\it polychromatic number} $p(S)$ of $S$ is defined as the largest number of colours for which such a colouring exists. It is not difficult to see that $d_b(S)\le 1/p(S)$. For a proof and more on polychromatic colourings of integers, see \cite{Axe}.

\section{Proof of Theorem~\ref{thm:-S} and Theorem~\ref{thm:3baton} for coverings}\label{sec:cov}

The following proposition solves the simplest case of Theorem \ref{thm:-S} when $\lambda_1 \equiv \lambda_2 \pmod 3.$

\begin{Proposition} \label{d l1=l2}
	Let $\lambda_1,\lambda_2\in \N$ be two coprime integers such that $\lambda_1 \equiv \lambda_2 \pmod 3.$ Then for $S=\{0,\lambda_1,\lambda_1+\lambda_2\}$, we have
	\begin{equation*}
		d_f(S,-S) = d_f(S) = \frac{2}{3}.
	\end{equation*}
\end{Proposition}
\begin{proof}
	On the one hand, it is easy to see that the set $A = \{0,1\}+3\Z$ avoids translates of both $S$ and $-S$. Thus, $\frac{2}{3}\le d_f(S,-S) \le d_f(S)$.
	
	On the other hand, it follows by a simple averaging and limiting argument that $d_f(S) \le \frac{2}{3}$. Indeed, assume that the set $A$ has the upper density $\overline{d}(A)>\frac 23$. Take a positive $\varepsilon$ such that $\overline{d}(A)>\frac{2}{3} + \varepsilon$. By definition, there are infinitely many $n\in\N$ such that $\frac{|A\cap[-n,n]|}{2n+1}>\frac{2}{3} + \varepsilon$. Take a random translate $S'$ of $S$ inside such $[-n,n]$. Then the expected intersection of $A$ with $S'$ has size strictly greater than $\frac{2}{3}|S| =2$ whenever $n$ is large enough in terms of $\varepsilon$ and $\textrm{diam}(S)$. In particular, there is a translate of $S$ that is contained in $A$, and thus $A$ is not $S$-free.
\end{proof}

Now we are left with two cases: $\lambda_1 \equiv \lambda_2+1 \pmod 3$ or $\lambda_1 \equiv \lambda_2-1 \pmod 3$. It is actually sufficient to consider only the former one. Indeed, if we are in the latter case, then replace $S$ with its reflection $-S$, and note that $d_f(S) = d_f(-S)$ and $d_f(S,-S) = d_f(-S,S)$. So, from now on we assume that $\lambda_1 \equiv \lambda_2+1 \pmod 3$. In it not hard to check that in this case Theorem~\ref{thm:-S} states that
\begin{equation*}
	d_f(S,-S)=d_f(S)= \frac{\lfloor \frac{2}{3}( 2\lambda_1+\lambda_2)\rfloor}{2\lambda_1+\lambda_2} = \frac{\frac{2}{3}(2\lambda_1+\lambda_2)-\frac{1}{3}}{2\lambda_1+\lambda_2}.
\end{equation*}

The next proposition provides the required lower bound on $d_f(S,-S)$ and thus on $d_f(S)$. The same constructions only for $d_f(S)$ were described in \cite{Schtu}. To check that they are indeed avoid translates of $-S$ as well, we recall them for completeness.

\begin{Proposition} \label{d l1=l2+1}
	Let $\lambda_1, \lambda_2 \in \N$ be two coprime integers such that $\lambda_1 \equiv \lambda_2+1 \pmod{3}$. Then for $S=\{0,\lambda_1,\lambda_1+\lambda_2\}$, we have
	\begin{equation*}
		d_f(S)\geq d_f(S,-S) \ge \frac{\lfloor \frac{2}{3}( 2\lambda_1+\lambda_2)\rfloor}{2\lambda_1+\lambda_2}.
	\end{equation*}
\end{Proposition}
\begin{proof} We say that a set $X\subset \Z_m$ is {\it cyclically $S$-free} if there are no three points $x,y,z\in X$  that satisfy $y-x \equiv \lambda_1\pmod{m},$ $z-y\equiv \lambda_2\pmod{m}$.
	We will show that there is a set $X\subset \Z_{2\lambda_1+\lambda_2}$ of cardinality $\left\lfloor \frac{2}{3}( 2\lambda_1+\lambda_2)\right\rfloor$ that is both cyclically $S$- and $(-S)$-free. This will imply the proposition since we can consider a periodic set $A = X+(2\lambda_1+\lambda_2)\Z$ of the required density that clearly avoids translates of both $S$ and $-S$.
	
	Note that $\gcd(\lambda_1, 2\lambda_1+\lambda_2) = \gcd(\lambda_1, \lambda_2) = 1$. Therefore, the additive subgroup $\langle \lambda_1\rangle \subset \Z_{2\lambda_1+\lambda_2}$ generated by the residue $\lambda_1$ is isomorphic to $\Z_{2\lambda_1+\lambda_2}$. Besides, we have $\lambda_2 \equiv -2\lambda_1 \pmod{2\lambda_1+\lambda_2}$. Hence, both cyclic translates of $S$ and $-S$ always cover three consecutive elements of $\langle \lambda_1\rangle$. 
	
	Now it is clear how to construct the desired set $X \subset \Z_{2\lambda_1+\lambda_2}$ that is both cyclically $S$- and $(-S)$-free: put $X = \{i\lambda_1: i\in I\}$, where $I\subset \Z_{2\lambda_1+\lambda_2}$ is a set that has no $3$ cyclically consecutive elements, e.g., $$I= \{0,1,\ 3,4,\ 6,7, \ \ldots \ , 2\lambda_1+\lambda_2-5,2\lambda_1+\lambda_2-4, \ 2\lambda_1+\lambda_2-2\}.$$ 
	It is not hard do check that $|X|=|I|=\left\lfloor \frac{2}{3}( 2\lambda_1+\lambda_2)\right\rfloor$, as required.
\end{proof}

It only remains  to prove the upper bound
\begin{equation} \label{eq:3batonupper}
	d_f(S) \le \frac{\lfloor \frac{2}{3}( 2\lambda_1+\lambda_2)\rfloor}{2\lambda_1+\lambda_2},
\end{equation}
where $S=\{0,\lambda_1,\lambda_1+\lambda_2\}$ and $\lambda_1, \lambda_2$ are coprime integers such that $\lambda_1 \equiv \lambda_2+1 \pmod{3}$.
The following two propositions incorporate a large part of the proof.

\begin{Proposition} \label{average}
	Let $A \subset \Z$ be an infinite set and $a,x,y \in \N$ be integers. Given $t \in \Z$, we define the interval $I(t) = [x]+t$ of length $x$ that begins at $t+1$. Assume that for all $t \in \Z$, we have either $|A\cap I(t)| \le a$, or $|A\cap I(t)|+|A\cap I(y+t)| \le 2a$. Then the upper density of $A$ satisfies $\overline{d}(A) \le \frac{a}{x}$.
\end{Proposition}
\begin{proof}	
	We argue indirectly. Assume that $\overline{d}(A) > \frac{a}{x}$. Then we fix an arbitrary $\varepsilon>0$ such that $\overline{d}(A) > \frac{a}{x}+\varepsilon$. By the definition of the upper density, there are infinitely many $n \in \N$ such that $|A\cap [-n,n]| > \left(\frac{a}{x}+\varepsilon\right)(2n+1)$. Putting $J\coloneqq[-n+x+y,n-x-y]$ for sufficiently large $n$, this implies that
	\begin{equation*}
		|A\cap J| \ge |A\cap [-n,n]|-2(x+y) > \left(\frac{a}{x}+\varepsilon\right)(2n+1) -2(x+y) > \left(\frac{2a}{x}+\varepsilon\right)n.
	\end{equation*}

	Given $z \in [y],i \in \N$, we set $J_i(z) = [x]-n-2-y+iy+z$. Choose $m$ to be the largest integer such that $J_{m}(y)\subset [-n,n]$, i.e., that $x+my \le 2n+2$. Note that $m\le \frac{2n+1}{y}$. We estimate the sum
	\begin{equation*}
		\sum_{z=1}^{y}\sum_{i=1}^{m} |A\cap J_i(z)|.
	\end{equation*}

	On the one hand, observe that any point of $A\cap J$ is a member of exactly $x$ intervals $J_i(z)$, $z \in [y], i \in [m]$. Hence,
	\begin{equation} \label{eq average lower}
		\sum_{z=1}^{y}\sum_{i=1}^{m} |A\cap J_i(z)| \ge x|A\cap J| > (2a+\varepsilon x)n \ge (2a+\varepsilon)n.
	\end{equation}
	
	On the other hand, we may fix $z \in [y]$ and apply the hypothesis of the proposition to $J_1(z)$ and $J_2(z)=J_1(z)+y$. We get that either $|A\cap J_1(z)|\le a$, or $|A\cap J_1(z)|+|A\cap J_2(z)|\le 2a$. In the former case, we apply the hypothesis of the proposition to $J_2(z),J_3(z)$, and in the latter case, we apply it to $J_3(z), J_4(z).$ Continuing in the same vein up while there are at least $2$ unvalued terms in the sum, we get that 
	\begin{equation*}
		\sum_{i=1}^{m} |A\cap J_i(z)|\le (m-1)a+x < ma+x.
	\end{equation*}
	Therefore,
	\begin{equation}  \label{eq average upper}
		\sum_{z=1}^{y}\sum_{i=1}^{m} |A\cap J_i(z)| < may+xy \le \left(\frac{2n+1}{y}\right)ay+xy = 2an+(a+xy) < (2a+\varepsilon)n,
	\end{equation}
	provided $n$ is large enough. 
	
	Finally, note that inequalities~\eqref{eq average lower} and~\eqref{eq average upper} contradict each other.
\end{proof}

\begin{Proposition}\label{prop:3batonupperv2}
	Let $\lambda_1,\lambda_2 \in \N$ be two coprime integers such that $\lambda_1 \equiv \lambda_2+1 \pmod{3}$. Set $S=\{0,\lambda_1,\lambda_1+\lambda_2\}$. Let $A\subset \mathbb{Z}$ be an $S$-free subset. For each $t\in \mathbb Z$, consider two intervals $I_1(t)=[2\lambda_1+\lambda_2]+t$ and $I_2(t)=[2\lambda_1+\lambda_2]+\lambda_1+\lambda_2+t$, each of length $2\lambda_1+\lambda_2$. Then $$2|I_1(t)\!\setminus\! A|+|I_2(t)\!\setminus\! A|\geq 2\lambda_1+\lambda_2.$$
\end{Proposition}
\begin{proof}
Since $A\subset \Z$ is $S$-free if and only if $A-t\subset \Z$ is $S$-free, it is sufficient to show the desired inequality only for $t=0$. In this case $I_1 \coloneqq I_1(0) = [2\lambda_1+\lambda_2]$ and $I_2 \coloneqq I_2(0) = [2\lambda_1+\lambda_2]+\lambda_1+\lambda_2$.

We start the proof by splitting $I_1\cup I_2$ into five disjoint intervals:
\begin{align*}
E_1 & = [1,\lambda_1],\\[3pt]
E_2&=[\lambda_1+1,\lambda_1+\lambda_2],\\[3pt]
E_3&=[\lambda_1+\lambda_2+1,2\lambda_1+\lambda_2],\\[3pt]
E_4&=[2\lambda_1+\lambda_2+1,3\lambda_1+\lambda_2],\\[3pt]
E_5&=[3\lambda_1+\lambda_2+1,3\lambda_1+2\lambda_2].
\end{align*}

For each $i \in [5],$ let us put $\alpha_i \coloneqq |E_i\!\setminus\! A|.$ Note that any element from $E_1,E_2,E_3,E_4,E_5$ is contained respectively in $1,2,3,2,1$ translates $S'\subset I_1 \cup I_2$ of $S$. Moreover, let us put $\alpha'_4 \coloneqq |X|$, where $X\subset E_4\!\setminus\! A$ is the set of the elements  $x\in E_4\!\setminus\! A$ such that $\{x-\lambda_1-\lambda_2,x-\lambda_1\}\subset A$, and put  $\alpha''_4 \coloneqq |(E_4\!\setminus\!A)\!\setminus\!X|=\alpha_4-\alpha'_4.$

For each translate $S'\subset I_1\cup I_2$ of $S$, we correspond the leftmost element of $S'\!\setminus\! A$. (Recall that $S'\!\setminus\! A$ is not an empty set since $A$ is $S$-free.) Note that every element $x\in E_4$ is contained in two such translates: $\{x-\lambda_1,x,x+\lambda_2\}$ and $\{x-\lambda_1-\lambda_2,x-\lambda_2,x\}$. Thus if $x\in E_4\!\setminus\!A$ is corresponding to both these translates, then we have $x \in X$ by construction.

Let us count, how many translates are possibly corresponding to an element from $E_i\!\setminus\!A$, for different $i.$ It should be clear that for $i=1,2,3,4,5$, an element from $E_i\!\setminus\!A$ is corresponding to at most $1,2,3,2,1$ translates, respectively. Moreover, recall that any element from $(E_4\!\setminus\!A)\!\setminus\!X$ is corresponding to at most one translate. Since there are $2\lambda_1+\lambda_2$ translates of $S$ inside $I_1\cup I_2$ in total, we get the following inequality:
$$\alpha_1+2\alpha_2+3\alpha_3+2\alpha'_4+\alpha''_4+\alpha_5\ge 2\lambda_1+\lambda_2.$$

Let us show that $\alpha_1\ge \alpha'_4.$ Indeed, if $x\in E_4\!\setminus\!A$ contributes to $\alpha'_4,$ then $\{x-\lambda_1-\lambda_2,x-\lambda_1\}\subset A$. Thus, $y(x) \coloneqq x-2\lambda_1-\lambda_2\notin A$ since $A$ is $S$-free. Note that $y(x)\in E_1\!\setminus\!A$ for any $x \in X$. 
Therefore, we constructed an injection from $X$ to $E_1\!\setminus\!A$, and thus $\alpha_1\ge \alpha_4'.$ Finally, we conclude that
\begin{align*}2|I_1\!\setminus\!A|+|I_2\!\setminus\!A|\ =&\ 2(\alpha_1+\alpha_2+\alpha_3)+ (\alpha_3+\alpha_4+\alpha_5)\\
\ \ge&\ \alpha_1+2\alpha_2+3\alpha_3+\alpha'_4+\alpha_4+\alpha_5\\
\ =&\ \alpha_1+2\alpha_2+3\alpha_3+2\alpha_4'+\alpha''_4+\alpha_5\\
\ \ge&\ 2\lambda_1+\lambda_2,
\end{align*}
as required.
\end{proof}

Proposition~\ref{prop:3batonupperv2} has the following corollary.

\begin{Corollary}\label{coraver}
    In the notation of Proposition~\ref{prop:3batonupperv2}, at least one of the following holds:    \begin{align*}
    	|A\cap I_1(t)|\ \le\ & \left\lfloor \frac{2}{3}( 2\lambda_1+\lambda_2)\right\rfloor;\\
		|A\cap I_1(t)|+|A\cap I_2(t)|\ \le\ 2&\left\lfloor \frac{2}{3}( 2\lambda_1+\lambda_2)\right\rfloor.
	\end{align*}
\end{Corollary}
\begin{proof}
    Assume that neither of the conclusions is valid. Then \begin{align*}
        2|I_1(t)\!\setminus\!A|+|I_2(t)\!\setminus\!A|\ =&\ 3(2\lambda_1+\lambda_2)-|A\cap I_1(t)|-(|A\cap I_1(t)|+|A\cap I_2(t)|)\\
        \ \le &\ 3(2\lambda_1+\lambda_2)-3\left\lfloor \frac{2}{3}( 2\lambda_1+\lambda_2)\right\rfloor-2\\
        \ =&\  3(2\lambda_1+\lambda_2) - 3\left(\frac{2}{3}( 2\lambda_1+\lambda_2)-\frac{1}{3}\right) - 2 \\
        \ =&\ 2\lambda_1+\lambda_2-1,
    \end{align*}
which contradicts the conclusion of Proposition~\ref{prop:3batonupperv2}.
\end{proof}

Now we finish the proof of the inequality~\eqref{eq:3batonupper}. 
Let $A\subset\Z$ be an $S$-free set. 
It follows from Corollary \ref{coraver} that $A$ satisfies the conditions of Proposition \ref{average} with
\begin{equation*}
	a = \left\lfloor \frac{2}{3}( 2\lambda_1+\lambda_2)\right\rfloor,  \ \ x = 2\lambda_1+\lambda_2, \ \ y = \lambda_1+\lambda_2.
\end{equation*}
Thus, 
Proposition~\ref{average} guarantees that
\begin{equation*}
	\overline{d}(A) \le \frac{\left\lfloor \frac{2}{3}( 2\lambda_1+\lambda_2)\right\rfloor}{2\lambda_1+\lambda_2}
\end{equation*}
and finishes the proof of the inequality~\eqref{eq:3batonupper} and hence of Theorem~\ref{thm:-S} as well.

Finally, recall that as we mentioned in the introduction, Theorem~\ref{thm:-S} together with the equality~\eqref{eq_cbf} yields the statement of Theorem~\ref{thm:3baton} for coverings.

\section{Proof of Theorem~\ref{thm:3baton} for packings}\label{sec:pack}

As in the proof of Theorem~\ref{thm:-S}, the case $\lambda_1 \equiv \lambda_2 \pmod 3$ is  trivial and the other two cases are `symmetric'. So, in what follows let us assume that $\lambda_1, \lambda_2 \in \N$ are coprime integers such that $\lambda_1 \equiv \lambda_2-1 \pmod 3$ and $S=\{0,\lambda_1,\lambda_1+\lambda_2\}$. 

The lower bound on $d_p(S)$ is via a direct example similar to the construction from Proposition~\ref{d l1=l2+1}: it is not hard to check that the periodic set
\begin{equation*}
	\Big\{i\lambda_1: 1 \le i \le 2\lambda_1+\lambda_2, \ i \equiv 0 \!\!\! \pmod{3}\Big\}+(2\lambda_1+\lambda_2)\Z
\end{equation*}
is $S$-packing with density $\frac{\lfloor \frac{1}{3}( 2\lambda_1+\lambda_2)\rfloor}{2\lambda_1+\lambda_2}$.

Thus, it remains only to show the upper bound.  We start by making a simple observation.

\begin{Lemma} \label{lemma:packing}
	Let $\lambda_1,\lambda_2 \in \N$ be two integers. Set $S=\{0,\lambda_1,\lambda_1+\lambda_2\}$. Let $A\subset \mathbb{Z}$ be an $S$-packing set. Then $A$ intersects each translate of $-S$ in at most one point.
\end{Lemma}
\begin{proof}
	Given $x \in \Z$, let us consider a translate $S'=\{x-\lambda_1-\lambda_2, x-\lambda_1, x\}$  of $-S$. Note that $x \in S+a$ for all $a \in S'$. Recall that translates $S+a_1$ and $S+a_2$ must be disjoint for all distinct $a_1,a_2 \in A$. Therefore, at most one element of $S'$ belongs to $A$.
\end{proof}

\begin{Proposition}\label{prop:3packingupper}
	Let $\lambda_1,\lambda_2 \in \N$ be two coprime integers such that $\lambda_1 \equiv \lambda_2-1 \pmod 3$. Set $S=\{0,\lambda_1,\lambda_1+\lambda_2\}$. Let $A\subset \mathbb{Z}$ be an $S$-packing set. For each $t\in \mathbb Z$, consider two intervals $I_1(t)=[2\lambda_1+\lambda_2]+t$ and $I_2(t)=[2\lambda_1+\lambda_2]+\lambda_1+\lambda_2+t$, each of length $2\lambda_1+\lambda_2$. Then
	\begin{equation*}
		2|I_1(t)\!\setminus\!A|+|I_2(t)\!\setminus\!A|\geq 4\lambda_1+2\lambda_2.
	\end{equation*}
\end{Proposition}
This proposition is very similar to Proposition~\ref{prop:3batonupperv2} and it was tempting for us to try and turn these two statements into one. We, however, did not succeed. The reason for this is that our partition of the domain is different in the two cases (and not symmetric under the transposition of $\lambda_1$ and $\lambda_2$).
\begin{proof}
	Since $A$ is $S$-packing if and only if $A-t$ is $S$-packing, it is sufficient to show the desired inequality only for $t=0$. In this case $I_1 \coloneqq I_1(0) = [2\lambda_1+\lambda_2]$ and $I_2 \coloneqq I_2(0) = [2\lambda_1+\lambda_2]+\lambda_1+\lambda_2$.

	We start the proof by splitting $I_1\cup I_2$ into five disjoint intervals:
	\begin{align*}
		E_1 & = [1,\lambda_2],\\[3pt]
		E_2&=[\lambda_2+1,\lambda_1+\lambda_2],\\[3pt]
		E_3&=[\lambda_1+\lambda_2+1,2\lambda_1+\lambda_2],\\[3pt]
		E_4&=[2\lambda_1+\lambda_2+1,2\lambda_1+2\lambda_2],\\[3pt]
		E_5&=[2\lambda_1+2\lambda_2+1,3\lambda_1+2\lambda_2].
	\end{align*}

	For each $i \in [5],$ let us put $\alpha_i \coloneqq |E_i\!\setminus\!A|.$ Note that any element from $E_1,E_2,E_3,E_4,E_5$ is contained respectively in $1,2,3,2,1$ translates $S'\subset I_1 \cup I_2$ of $-S$. Moreover, let us put $\alpha'_4 \coloneqq |X|$, where $X$ is the set of all elements  $x\in E_4\!\setminus\!A$ such that either $x-\lambda_1-\lambda_2 \in A$, or $x-\lambda_1 \in A$. Set $\alpha''_4 \coloneqq |(E_4\!\setminus\!A)\!\setminus\!X|=\alpha_4-\alpha'_4.$
	
	For each translate $S'\subset I_1\cup I_2$ of $-S$, let us correspond two of its elements that do not belong to $A$. Lemma \ref{lemma:packing} ensures that it is always possible. For most of the translates, we do it arbitrarily, the only exception being when $S'=\{x-\lambda_1-\lambda_2, x-\lambda_1, x\},$ where $x \in (E_4\!\setminus\!A)\!\setminus\!X$: to such $S'$ we correspond $x-\lambda_1-\lambda_2$ and $x-\lambda_1$.
	
	Let us count, how many translates of $-S$ are possibly corresponding to an element from $E_i\!\setminus\!A$, for different $i.$ For $i=1,2,3,4,5$ an element from $E_i\!\setminus\!A$ is trivially corresponding to at most $1,2,3,2,1$ translates, respectively. Besides, observe that an element $x \in (E_4\!\setminus\!A)\!\setminus\!X$ is corresponding to at most one translate. Indeed, it is not corresponding to $\{x-\lambda_1-\lambda_2, x-\lambda_1, x\}$ by construction, so, $\{x-\lambda_2, x, x+\lambda_1 \}$ is the only possible option here. Since there are $2\lambda_1+\lambda_2$ translates of $-S$ inside $I_1\cup I_2$ in total and we correspond two of its points to each of them, we get the following inequality:
	\begin{equation*}
		\alpha_1+2\alpha_2+3\alpha_3+2\alpha'_4+\alpha''_4+\alpha_5\ge 4\lambda_1+2\lambda_2.
	\end{equation*}
	
	Let us show that $\alpha_1\ge \alpha'_4.$ Given $x\in X$, we consider $y(x)=x-2\lambda_1-\lambda_2 \in E_1$. Note that both $x-\lambda_1-\lambda_2$ and $x-\lambda_1$ belong to $S+y(x)$. Since
	one of these two elements belongs to $A$ by the definition of $X$ and since $A$ is $S$-packing, we conclude that $y(x)$ does not belong to $A$. Therefore, we constructed an injection from $X$ to $E_1\!\setminus\!A$, and thus $\alpha_1\ge \alpha_4'.$ Finally, we conclude that
	
	\begin{align*}2|I_1\!\setminus\!A|+|I_2\!\setminus\!A|\ =&\ 2(\alpha_1+\alpha_2+\alpha_3)+(\alpha_3+\alpha_4+\alpha_5)\\
		\ \ge&\ \alpha_1+2\alpha_2+3\alpha_3+\alpha'_4+\alpha_4+\alpha_5\\
		\ =&\ \alpha_1+2\alpha_2+3\alpha_3+2\alpha_4'+\alpha''_4+\alpha_5\\
		\ \ge&\ 4\lambda_1+2\lambda_2,
	\end{align*}
	as required.
\end{proof}

\begin{Corollary}\label{cor:packing}
	In the notation of Proposition~\ref{prop:3packingupper}, at least one of the following holds: 
	\begin{align*}
		|A\cap I_1(t)|\ \le\ & \left\lfloor \frac{1}{3}( 2\lambda_1+\lambda_2)\right\rfloor;\\
		|A\cap I_1(t)|+|A\cap I_2(t)|\ \le\ 2&\left\lfloor \frac{1}{3}( 2\lambda_1+\lambda_2)\right\rfloor.
	\end{align*}
\end{Corollary}
\begin{proof}
	Assume that neither of the conclusions is valid. Then \begin{align*}
		2|I_1(t)\!\setminus\!A|+|I_2(t)\!\setminus\!A|\ =&\ 3(2\lambda_1+\lambda_2)-|A\cap I_1(t)|-(|A\cap I_1(t)|+|A\cap I_2(t)|)\\
		\ \le &\ 3(2\lambda_1+\lambda_2)-3\left\lfloor \frac{1}{3}( 2\lambda_1+\lambda_2)\right\rfloor-2\\
		\ =&\  3(2\lambda_1+\lambda_2) - 3\left(\frac{1}{3}( 2\lambda_1+\lambda_2)-\frac{1}{3}\right) - 2 \\
		\ =&\ 4\lambda_1+2\lambda_2-1,
	\end{align*}
	which contradicts the conclusion of Proposition~\ref{prop:3packingupper}.
\end{proof}

Now we finish the proof of Theorem~\ref{thm:3baton} for packings. Let $\lambda_1,\lambda_2 \in \N$ be two coprime integers such that $\lambda_1 \equiv \lambda_2-1 \pmod 3,$ $S=\{0,\lambda_1,\lambda_1+\lambda_2\}$, and $A\subset\Z$ be an $S$-packing set. It follows from Corollary~\ref{cor:packing} that $A$ satisfies the conditions of Proposition \ref{average} with
\begin{equation*}
	a = \left\lfloor \frac{1}{3}( 2\lambda_1+\lambda_2)\right\rfloor,  \ \ x = 2\lambda_1+\lambda_2, \ \ y = \lambda_1+\lambda_2.
\end{equation*}
Thus, Proposition \ref{average} guarantees that
\begin{equation*}
	\overline{d}(A) \le \frac{\left\lfloor \frac{1}{3}( 2\lambda_1+\lambda_2)\right\rfloor}{2\lambda_1+\lambda_2}.
\end{equation*}
Since the last value bounds from above the upper densities $\overline{d}(A)$ of all $S$-packing subsets $A\subset\Z$, it bounds from above their supremum $d_p(S)$ as well. This finally completes the proof of Theorem~\ref{thm:3baton}. 



{\small \begin{thebibliography}{20}
		
\bibitem{Axe} M. Axenovich, J. Goldwasser, B. Lidick\'y, R.R. Martin, D. Offner, J. Talbot, M. Young, {\it Polychromatic Colorings on the Integers}, Integers, {\bf 19} (2019), A18.

\bibitem{Bhat2020} S. Bhattacharya, {\it Periodicity and decidability of tilings of $\Z^2$}, Amer. J. Math., {\bf 142} (2020), N1, 255--266.

\bibitem{BJR} B. Bollob\'as, S. Janson, O. Riordan, {\it On covering by translates of a set}, Random Struct. Alg., {\bf 38} (2011), 33--67. 

\bibitem{BMP} P. Brass, W.O. Moser, J. Pach, {\it Research problems in discrete geometry}, Springer, New York, NY, 2005.

\bibitem{CM} E.M. Coven, A. Meyerowitz, {\it Tiling the integers with translates of one finite set}, J. Algebra, {\bf 212} (1999), 161--174.

\bibitem{Erd54} P. Erd\H os, {\it Some results on additive number theory}, Proc. Amer. Math. Soc., {\bf 5} (1954), 847--853.

\bibitem{FKS} N. Frankl,  A. Kupavskii, A. Sagdeev, {\it Max-norm Ramsey Theory}, arXiv:2111.08949.

\bibitem{Gol} M.J.E. Golay, {\it Notes on the representation of $1, 2, \dots, n$ by differences}, J. Lond. Math. Soc., {\bf 2} (1972), N4, 729--x734.
 
\bibitem{KS}  A. Kupavskii, A. Sagdeev, {\it All finite sets are Ramsey in the maximum norm}, Forum Math. Sigma, {\bf 9}~(2021), e55.

\bibitem{KS2} A. Kupavskii, A. Sagdeev, {\it Ramsey theory in the $n-$space with Chebyshev metric}, Russian Math. Surveys, {\bf 75} (2020), N5, 965--967.

\bibitem{LZ04} D.D.F. Liu, X. Zhu, {\it Fractional chromatic number and circular chromatic number for distance graphs with large clique size}, J. Graph Theory, {\bf 47} (2004), N2, 129--146.

\bibitem{Lor} G.G. Lorentz, {\it On a problem of additive number theory}, Proc. Amer. Math. Soc., {\bf 5} (1954), 838--841.

\bibitem{New} D.J. Newman, {\it Complements of finite sets of integers}, Michigan Math. J., {\bf 14} (1967), N4, 481--486.

\bibitem{New2} D.J. Newman, {\it Tesselation of integers}, J. Number Theory, {\bf 9} (1977), N1, 107--111.

\bibitem{RabPr} J.H. Rabinowitz, V.K. Proulx, {\it An asymptotic approach to the channel assignment problem}, SIAM J. Algebr. Discrete Methods, {\bf 6} (1985), N3, 507--518.

\bibitem{Schtu} W.M. Schmidt, D.M. Tuller, {\it Covering and packing in ${\Z}^n$ and ${\R}^n$, (I)}, Monatsh. Math., {\bf 153} (2008), N3, 265--281. 

\bibitem{SchTul2010} W.M. Schmidt, D.M. Tuller, {\it Covering and packing in $\Z^n $ and $\R^n $, (II)}, Monatsh. Math., {\bf 160} (2010), N2, 195--210.

\bibitem{Ste} S.K. Stein, {\it Tiling, packing and covering by clusters}, Rocky Mountain J. Math., {\bf 16} (1986), 277--321.

\bibitem{Wein} G. Weinstein, {\it Some covering and packing results in number theory},  J. Number Theory, {\bf 8} (1976), 193--205.

\end{thebibliography}}
\end{document}